\documentclass{tran-l}

\newtheorem{thm}{Theorem}[subsection]

\newtheorem{lem}[thm]{Lemma}
\newtheorem{prop}[thm]{Proposition}
\theoremstyle{definition}
\newtheorem{defn}[thm]{Definition}
\theoremstyle{remark}
\newtheorem{rem}[thm]{Remark}
\newtheorem{exam}[thm]{Example}
\numberwithin{equation}{subsection}

\newcommand{\h}{\mathcal{H}}

\newcommand{\A}{\mathcal{A}}
\newcommand{\J}{\mathcal{J}}

\newcommand{\Field}{\mathbb{F}}

\newcommand{\Lom}{\mathcal{L}}

\usepackage{graphicx}
\usepackage{setspace}

\begin{document}

\title[CONSTRUCTION OF SOME NON-ASSOCIATIVE ALGEBRAS]
 {Construction of some non-associative algebras from Associative Algebras with an endomorphism
operator, a differential operator or a left averaging operator.}

\author{Wilson Arley Martinez, Samin Ingrith Ceron}

\address{Martinez, W.A.; Departmento de Matem\'aticas, Universidad del Cauca, Popay\'an, Colombia}
\email{wamartinez@unicauca.edu.co}

\address{Ceron, S.I.; Departmento de Matem\'aticas, Universidad del Cauca, Popay\'an, Colombia}
\email{sicbravo@gmail.com}

\thanks{This work was completed with the support of the Universidad del Cauca.}

\thanks{The authors was also supported by the research group “Estructuras Algebraicas, Divulgación Matemática y Teorías Asociadas. @DiTa”.}




\date{January 27, 2023 and, in revised form, January, 2023.}

\dedicatory{}

\commby{W.A.M}


\begin{abstract}
In this paper, we introduce the concepts of endomorphism operator, left averaging operator, differential operator and Rota-Baxter Operator, and we construct examples of these linear maps on associative algebras with a left identity, a skew-idempotent or an idempotent element. These maps on associative algebra induce a non-associative algebra structure such as Lie algebra, Pre-Lie algebra, Jordan algebra, Flexible Algebra or (left) Leibniz algebra. We consider the construction of non-associative algebras from associative algebras with Linear Operators as the main results of this work. In this paper we give examples of non-associative algebras on subspaces of square matrices $ M_{3}(\mathbb{R})$.
\end{abstract}

\maketitle

\section*{Introduction}

Linear operators can be defined on different algebraic structures, the well-known operators are the endomorphism operator and yhe differential operator ~\cite{Kolchin1985, van2003, Li2013}. By the 1970's, new identities for operators emerged from studies in combinatorics, probability and analysis. Gian-Carlo Rota was most interested in the following operators:

\begin{center}
\begin{tabular}{rl}
{\it Endomorphism operator}    & $R(x\cdot y) = R(x)\cdot R(y),$\vspace{0.3cm}\\
{\it Differential operator}    & $R(x\cdot y) = R(x)\cdot y + x\cdot R(y),$\vspace{0.3cm} \\
{\it Rota-Baxter operator of weight $\lambda$,}  &                                              \vspace{0cm}\\
{\it where $\lambda$ is a fixed constant}       & $R(x)\cdot R(y) = R(x\cdot R(y) + R(x)\cdot y + \lambda x\cdot y),$ \vspace{0.3cm}\\
{\it Average operator}          &  $R(x)\cdot R(y) = R(x\cdot R(y)),$\vspace{0.3cm} \\
{\it Inverse average operator}  & $R(x)\cdot R(y) = R(R(x)\cdot y),$\vspace{0.3cm} \\
{\it Reynolds operator}         & $R(x)\cdot R(y) = R(x\cdot R(y) + R(x)\cdot y - R(x)\cdot R(y)).$
\end{tabular}
\end{center}

\medskip

An endomorphism is a homomorphism from an algebraic structure into itself. Let $A$ be a non-unital, associative algebra, $\alpha$ an algebra endomorphism on $A$ and define $\ast: A\times A \longrightarrow A$ by $a\ast b = \alpha(a·b)$ for all $a,b \in A$. Then $(A,\ast,\alpha)$ is a hom-associative algebra, see ~\cite{Yau2009}.

\medskip

The study of averaging operators from an algebraic point of view was started by Kampé de Fériet ~\cite{KAMPE1949} and continued
and elaborated by Birkhoff ~\cite{BIRKHOFF1949}. Averaging operators have connection with developments in the theory of turbulence ~\cite{BIRKHOFF1949,BIRKHOFF1960}, and are closely related to the probability theory.

\medskip

In his Ph. D. thesis in 2000 ~\cite{Weili2014}, Weili Cao studied averaging operators in the general context and the algebraic definition. He studied the naturally induced Lie algebra structures from averaging operators: Let $R : A \rightarrow A$  be an Averaging Operator on an algebra $A$, this map permits us to define a Lie bracket operation on $A$, by  $[x,y] = x\cdot R(y) - y\cdot R(x)$,  $\forall x,y \in A$, see ~\cite{Weili2014,Nguyen1976}.

\medskip

Let $R : A \rightarrow A$  be a Differential Operator on a commutative associative algebra $A$. This map induces a new Lie algebra structure called Witt type Lie algebras ~\cite{Xu2000}, defined by the bracket $[x,y] = R(x)\cdot y - x\cdot R(y)$,  $\forall x,y \in A$. Commutative associative algebras with this type of linear maps permit us to present examples of Lie algebras.

\medskip

Rota-Baxter operators were introduced by the mathematician Glenn E. Baxter ~\cite{Baxter1960}, in the study of differential equations applied to probability theory, and its importance comes mainly from by the works of G.-C. Rota in combinatorics ~\cite{Baxter1969,Rota1995,Rota1972}.

\medskip

A Rota-Baxter algebra is an associative algebra equipped with a Rota-Baxter operator. Recently, noncommutative Rota-Baxter algebras have appeared in a wide range of areas in pure mathematics, for example the works of Loday and Ronco on dendriform dialgebras and trialgebras, see ~\cite{Lo1,LodayRonco2002} and in applied mathematics too, see ~\cite{Connes2000}.

\medskip

The following result provides a way to construct a pre-Lie algebra structure from a Rota-Baxter operator relation on Lie Algebras or pre-Lie algebras. We find that if $R : A \rightarrow A$  is a Rota Baxter-Operator on an Lie Algebra $(A, [,])$, this map induces a pre-Lie algebra structure, defined by $ x\ast y =[R(x),y] $,  $\forall x,y \in A$, see ~\cite{Huihui2008}.

\medskip

In the case of the Rota-Baxter relation on pre-Lie algebras, it is known that if $(A,\cdot)$ is a pre-Lie algebra and $R$ is a Rota-Baxter operator on $A$ then $R$ is still a Rota-Baxter operator on $(A, \ast)$, and the product given by  $x\ast y = [R(x), y]$ $ = R(x) \cdot y - y \cdot R(x), x, y \in A$, defines a new pre-Lie algebra $(A, \ast)$, see ~\cite{Xiuxian2007} .

\medskip

An element $u$ is said to be skew-idempotent with respect to a product · in the algebra if $u \cdot u = -u$, and an element $u$ is a right identity if $x \cdot u = x$ for all element $x$ in the algebra. Associative algebras with a left identity, a skew-idempotent or an idempotent element permit us to build examples of linear maps such as Endomorphism Operator, Left Averaging Operator, Differential Operator and Rota-Baxter Operator. Associative algebras with this type of linear maps permit us to present constructions of non-associative algebras.

\section{Construction of Lie algebras from Associative Algebras with an Endomorphism Operator}

In this section, we present in Proposition ~\ref{p:LieAlge} a Lie algebra structure given by the following bracket $[x,y]=x\cdot R(y)-y\cdot R(x)$ where $R$ is an endomorphism operator, and $R$ is defined from a {\it right identity on a subalgebra $\A$}:  Let $(\A, \, \cdot )$ be an algebra, then an element $u$ of $\h$, $\A\subseteq \h$,  is called a {\it right identity on $\A$} if $x\cdot u  = x$ for all $x$ in $\A$.

\subsection{Introduction}
In Definition \ref{d:identity} we define a left and a right identity for an associative algebra $\A$  and in the Proposition \ref{p:Endo} we give a proof of the construction of an endomorphism operator with certain properties on $\A$ using a right identity, inspired by the well known result of  X. Xu, \cite{Xu2000}, where he proved the existence of Lie Algebras from an associative algebra $\A$ with an Differential Operator on the space $\A$. We establish a new connection between an endomorphism operator with a construction of Lie algebra structures on $\A$ . We start by briefly introducing the definition of a Lie Algebra, a left and a right identity for $\A$.

\begin{defn}
A Lie algebra over a field $\Field$ is a vector space $\mathfrak{g}$ over $\Field$ equipped with a bilinear operation $[,]:\mathfrak{g}\times \mathfrak{g}\rightarrow \mathfrak{g}$, called the (Lie) bracket which satisfies the following identities:
\begin{eqnarray}
  [x,y] &=& -[y,x] \hspace{0.5cm} (\text{Antisymmetry}),  \\
  \left[x,[y,z]\right]+\left[z,[x,y]\right]+\left[y,[z,x]\right] &=& 0 \hspace{1.3cm} (\text{Jacobi identity}).
\end{eqnarray}
\end{defn}

\begin{rem}
It is well known that any associative algebra becomes a Lie algebra with the Lie bracket given by the commutator: $[x,y]=x\cdot y-y\cdot x$. Also, that the dimension of a Lie algebra $\mathfrak{g}$ is its dimension as a vector space over $\Field$ and Ado's theorem states that every finite dimensional Lie algebra $\mathfrak{g}$ over a field $\Field$ can be viewed as a Lie algebra of square matrices with the commutator as bracket.
\end{rem}

\begin{prop} \label{p:LieAlge}
Let $\A$ be an associative algebra and let $R :\A\rightarrow  \A$ be a linear
map such that $R^{2}(x)=R(x)$  and  $R(x)\cdot R(y) = R(x\cdot y)$ for all $x, y \in \A.$
Then we can define a Lie algebra structure on $\A$ given by
\begin{equation}\label{Ec:struli}
[x,y]=x\cdot R(y)-y\cdot R(x) \hspace{0.2cm} (\text{respectively } [x,y]=R(x)\cdot y-R(y)\cdot x )
\end{equation}
\end{prop}

\begin{proof}
Let $x,y,z\in \A$; then, $[x,y]=-[y,x]$  and
\begin{align*}
[x,[y,z]]& = x\cdot R([y,z])-[y,z]\cdot R(x) \\
         & = x\cdot R(y\cdot R(z)-z\cdot R(y))-( y\cdot R(z)-z\cdot R(y) )\cdot R(x)\\
         & = x\cdot (R(y)\cdot R(z)-R(z)\cdot R(y))-( y\cdot R(z))\cdot R(x) + (z\cdot R(y) )\cdot R(x)\\
         & = x\cdot (R(y)\cdot R(z))-x\cdot (R(z)\cdot R(y))-( y\cdot R(z))\cdot R(x) + (z\cdot R(y) )\cdot R(x)\\
[y,[z,x]]& = y\cdot R([z,x])-[z,x]\cdot R(y) \\
         & = y\cdot R(z\cdot R(x)-x\cdot R(z))-(z\cdot R(x)-x\cdot R(z))\cdot R(y)\\
         & = y\cdot (R(z)\cdot R(x)-R(x)\cdot R(z))-(z\cdot R(x))\cdot R(y)+(x\cdot R(z))\cdot R(y)\\
         & = y\cdot (R(z)\cdot R(x))-y\cdot (R(x)\cdot R(z))-(z\cdot R(x))\cdot R(y)+(x\cdot R(z))\cdot R(y)\\
[z,[x,y]]& = z\cdot R([x,y])-[x,y]\cdot R(z) \\
         & = z\cdot R(x\cdot R(y)-y\cdot R(x))-(x\cdot R(y)-y\cdot R(x))\cdot R(z)\\
         & = z\cdot (R(x)\cdot R(y)-R(y)\cdot R(x))-(x\cdot R(y))\cdot R(z)+(y\cdot R(x))\cdot R(z)\\
         & = z\cdot (R(x)\cdot R(y))-z\cdot (R(y)\cdot R(x))-(x\cdot R(y))\cdot R(z)+(y\cdot R(x))\cdot R(z)
\end{align*}
Thus, $[x,[y,z]]+[y,[z,x]]+[z,[x,y]] =0.$ Therefore, $(A,\,[\,, ])$ is a Lie algebra.
\end{proof}

\subsection{Examples of Lie algebras from the endomorphism operator.}

In this subsection we present the relationship between an endomorphism operator and a right identity of an associative algebra. In one direction, we show that an associative algebra $\A$ with a right identity gives an endomorphism operator with certain properties on the associative algebra $\A$. This allow us to give examples of lie algebras from the endomorphism.

\begin{defn}\label{d:identity}
Let $\A$ be an algebra and $\h$  a set containing $\A$. An element $u\in \h$ is called a \emph{left identity for $\A$} if $u\cdot x=x$ for all $x\in \A$. Similarly, $u\in \h$ is a  \emph{right identity for $\A$} if $x\cdot u=x$ for all $x\in \A$. An element $u\in \A$ which is both a left and a right identity for $\A$ is an identity element.
\end{defn}

\begin{rem}
Given an operation (function) $\cdot :\h \times \A \rightarrow \A$, an element $u$ of $\h$ is called a right identity for $\cdot$ if $x\cdot u=x$ for every element $x$ of $\A$. That is, the map $\A\rightarrow\A$ given by $x\cdot u$ is the identity function on $\A$.
\end{rem}

The following proposition and associated examples introduce the idea for all the main results of this paper.

\begin{prop} \label{p:Endo}
Let $( \, \A, \, \cdot \, )$ be an associative algebra and $\h$ a set containing $\A$, and suppose that there exists $u\in \h$ such that $u\cdot x  \in \A$ and $x\cdot u=x$ for all $x\in \A$. Then the linear map $R: \A\longrightarrow \A$ defined by $R(x)=u\cdot x$ satisfies
\begin{equation}\label{Ec:righid}
R(x)\cdot R(y)=R(x\cdot y) \text{ for all } x,y \in \A.
\end{equation}
Furthemore, $R^{2}(x)=R(x)$ if $u^{2}=u$, or $R^{2}(x)=x$ if $u^{2}=1$.
\end{prop}

\begin{proof}
Let $x,y \in \A$, then $R(x\cdot y) = u\cdot (x\cdot y).$ On the other hand, we have
$R(x)\cdot R(y)=(u\cdot x)\cdot(u\cdot y)=u\cdot((x\cdot u)\cdot y)=u\cdot(x\cdot y).$
Therefore, $R(x)\cdot R(y) = R(x\cdot y)$ for all $x, y \in \A.$ Now, if $u^{2}=u$,
then $R^{2}(x)=u\cdot(u\cdot x)=(u^{2}\cdot x)=u\cdot x=R(x).$
Therefore, $R^{2}(x)=R(x)$ for all $x\in \A.$
\end{proof}

\begin{exam}
We consider the subalgebra
$$\A=\left\{\left(\begin{array}{ccc}
                   y & y & 0 \\
                   n & n & 0 \\
                   r & r & 0 \\
                 \end{array}\right): r, n , y \in \mathbb{R}  \right\}$$  under the usual matrix multiplication.

The element $u=\left(\begin{array}{ccr}
                   a    &  b   &  c \\
                   1-a  &  1-b & -c \\
                   e    &  f   &  g\\
                 \end{array}\right)$ satisfies $x\cdot u=x$ and $u\cdot x\in \A$ for all $x\in \A$.

Then the linear map $R:\A\longrightarrow \A$ defined by

\begin{align*}
R\left( \left(\begin{array}{rrr}
                   y   &  y  & 0 \\
                   n   &   n & 0 \\
                   r   &   r  & 0\\
                 \end{array}\right) \right)
                 & =\left(\begin{array}{ccr}
                   a    &  b    & c \\
                   1-a  &  1-b  & -c \\
                    e   &  f    &  g\\
                 \end{array}\right)\cdot \left(\begin{array}{rrr}
                   y    &  y  & 0 \\
                   n   &   n & 0 \\
                   r   &   r  & 0\\
                 \end{array}\right)      \\
                 & =\left(\begin{array}{ccc}
                   ay+bn+cr           &  ay+bn+cr        & 0 \\
                   (1-a)y+(1-b)n-cr   & (1-a)y+(1-b)n-cr & 0 \\
                   ey+fn+gr           &   ey+fn+gr       & 0\\
                 \end{array}\right)
\end{align*}
satisfies $R(x)\cdot R(y) = R(x\cdot y)$ for all $x, y \in \A.$
\end{exam}

\begin{exam}

We consider the subalgebra
$$A=\left\{\left(\begin{array}{ccc}
                   x & y & y \\
                   w & k & k \\
                   m & n & n \\
                 \end{array}\right): x, w , m, y, k, n \in \mathbb{R}  \right\}$$  under the usual matrix multiplication.

The element $u=\left(\begin{array}{ccr}
                   1    &  0   &  0 \\
                   0    &  0   &  1 \\
                   0    &  1   &  0\\
                 \end{array}\right)$ satisfies $u^{2}=I$, \,  $x\cdot u=x$ and $u\cdot x\in A$ for all $x\in \A$.

Then the linear map $R:\A\longrightarrow \A$ defined by

$R\left( \left(\begin{array}{rrr}
                   x    &  y  & y \\
                   w   &   k  & k \\
                   m   &   n  & n\\
                 \end{array}\right) \right)=\left(\begin{array}{ccr}
                    x   &  y    & y \\
                    m   &  n    & n \\
                    w   &  k    & k\\
                 \end{array}\right)$
\medskip

satisfies  $R^{2}(x)=x$  and  $R(x)\cdot R(y) = R(x\cdot y)$ for all $x, y \in \A.$
\end{exam}

\begin{exam}

We consider the subalgebra
$$A=\left\{\left(\begin{array}{ccc}
                   y & y & 0 \\
                   n & n & 0 \\
                   r & r & 0 \\
                 \end{array}\right): y,n,r \in \mathbb{R}  \right\}$$  under the usual matrix multiplication.

The element $u=\left(\begin{array}{ccr}
                   1    &  b   &  b \\
                   0    &  1-b & -b \\
                   0    &  b-1 &  b\\
                 \end{array}\right)$ satisfies $u^{2}=u$, \,  $x\cdot u=x$ and $u\cdot x\in \A$ for all $x\in \A$.

Then the linear map $R:\A\longrightarrow \A$ defined by

\begin{align*}
R\left( \left(\begin{array}{rrr}
                   y   &   y  & 0 \\
                   n   &   n  & 0 \\
                   r   &   r  & 0\\
                 \end{array}\right) \right)
                 &=\left(\begin{array}{ccr}
                   1    &  b    & b \\
                   0    &  1-b  & -b \\
                   0    &  b-1  &  b\\
                 \end{array}\right)\cdot \left(\begin{array}{rrr}
                   y   & y  & 0 \\
                   n   &  n & 0 \\
                   r   &  r  & 0\\
                 \end{array}\right)\\
                 &=\left(\begin{array}{ccc}
                   y+bn+br   &  y+bn+br       & 0 \\
                   n-bn-br   & n-bn-br        & 0 \\
                   nb-n+br   & nb-n+br        & 0\\
                 \end{array}\right)
\end{align*}
satisfies $R^{2}(x) =R(x)$ and $R(x)\cdot R(y) = R(x\cdot y)$ for all $x, y \in \A.$ Therefore, we can define a Lie algebra structures on $\A$ given by $[x,y]=x\cdot R(y)-y\cdot R(x)$.
\end{exam}

\section{Construction of Jordan algebras from Commutative Associative Algebras with an Endomorphism Operator}
Jordan algebras were introduced in the early 1930’s by a physicist, P. Jordan, in an attempt to generalize the formalism of quantum mechanics. Little appears to have resulted in this direction, but unanticipated relationships between these algebras and Lie groups and the foundations of geometry have been discovered.

\begin{defn}
A (non-commutative) Jordan algebra is a vector space $\J$ over a field $\Field$ of characteristic $\neq 2$  with a binary operation $\circ$ satisfying for $x,y\in \J$ the following identity:
\begin{equation}\label{Ec:Jordan}
(x\circ y)\circ x = x\circ (y\circ x)  \quad \text{ and } \quad ( (x\circ x)\circ  y ) \circ  x = (x\circ x) \circ (y\circ x).
\end{equation}
\end{defn}

\begin{rem}
Given an associative algebra $(\A,\cdot)$ we can modify the product to obtain a commutative algebra $\A^{+}$ as follows: To construct $\A^{+}$ we define $x\ast y = x\cdot y + y\cdot x$. In this new algebra the Jordan identity is satisfied:
\begin{equation}\label{e:EssIE}
((x \ast  x) \ast  y) \ast  x = (x \ast  x) \ast  (y \ast  x).
\end{equation}
\end{rem}

\begin{exam}
We consider the subalgebra
$$\A^{+}=\left\{\left(\begin{array}{ccc}
                   y & y & 0 \\
                   n & n & 0 \\
                   r & r & 0 \\
                 \end{array}\right): y,n,r \in \mathbb{R}  \right\}$$  under the product $x\ast y = x\cdot y + y\cdot x$ is a commutative algebra (non-associative),  where $\cdot$ is the usual matrix multiplication.

The element $u=\left(\begin{array}{ccr}
                   1    &  b   &  b \\
                   0    &  1-b & -b \\
                   0    &  b-1 &  b\\
                 \end{array}\right)$ satisfies $u^{2}=u$, \,  $x\cdot u=x$ and $u\cdot x\in \A^{+}$ for all $x\in \A^{+}$.

Then the linear map $R:\A^{+}\longrightarrow \A^{+}$ defined by
\begin{align*}
R\left( \left(\begin{array}{rrr}
                   y   &  y  & 0 \\
                   n   &  n  & 0 \\
                   r   &  r  & 0\\
                 \end{array}\right) \right)
                 &=\left(\begin{array}{ccr}
                   1    &  b    & b \\
                   0    &  1-b  & -b \\
                   0    &  b-1  &  b\\
                 \end{array}\right)\cdot \left(\begin{array}{rrr}
                   y    & y  & 0 \\
                   n   &  n & 0 \\
                   r   &  r  & 0\\
                 \end{array}\right)\\
                 &=\left(\begin{array}{ccc}
                   y+bn+br   &  y+bn+br       & 0 \\
                   n-bn-br   & n-bn-br        & 0 \\
                   nb-n+br   & nb-n+br        & 0\\
                 \end{array}\right)
\end{align*}
satisfies $R^{2}(x) =R(x)$ and $R(x)\ast R(y) = R(x\ast y)$ for all $x, y \in \A^{+}.$ Therefore, the Jordan identity is satisfied on $\A^{+}$, with the product given by $x\circ y=R(x)\ast R(y)$.
\end{exam}

\begin{prop}\label{p:ESSPECT}
Suppose $(\A, \cdot)$ is a Jordan algebra, and $R : \A \rightarrow  \A$ is a linear
map, such that
\begin{equation} \label{e:ESSINEQ}
  R^{2}(x)=R(x) \text{  y  } R(x)\cdot R(y) = R(x\cdot y) \text{ for all x, y } \in \A .
\end{equation}
Then we can define a new (non-commutative) Jordan algebra structure on $\A$ given by $$x\circ y=R(x)\cdot y  \hspace{0.5cm} ( \text{respectively,  } x\circ y=x\cdot R(y), \,\, x\circ y=R(x)\cdot R(y)) .$$
\end{prop}

\begin{proof}
Let $\A$  be a Jordan algebra and let $R : \A \rightarrow  \A$ be a linear
map such that  $R^{2}(x)=R(x) \text{  y  } R(x)\cdot R(y) = R(x\cdot y)$ for all $x, y \in \A.$ So
\begin{align*}
(x\circ x) \circ (y \circ x) & = R(R(x)\cdot x)\circ (R(y)\cdot x ) \\
                             & = (R(x)\cdot R(x))\cdot (R(y)\cdot x)\\
                             & = ((R(x)\cdot R(x))\cdot R(y))\cdot x \\
                             & = ((R^{2}(x)\cdot R^{2}(x))\cdot R(y))\cdot x \\
                             & = R((R(x)\cdot R(x))\cdot y)\cdot x \\
                             & = R((R^{2}(x)\cdot R(x))\cdot y)\cdot x \\
                             & = R(R(R(x)\cdot x)\cdot y)\cdot x \\
                             & = ((x \circ x)\circ y)\circ x .
\end{align*}
This means that $(x\circ x) \circ (y \circ x)= ((x \circ x)\circ y)\circ x $ for all $x, y \in \A.$
Since
\begin{align*}
(x\circ y)\circ x   = R(R(x)\cdot y)\cdot x
                    = (R(x)\cdot R(y))\cdot x
                    = R(x)\cdot ( R(y)\cdot x)
                    = x\circ(y\circ x),
\end{align*}
then $(x\circ y)\circ x = x\circ (y\circ x) $ for all $x, y \in \A.$ Therefore, $(\A, \circ)$ is a Jordan algebra.
\end{proof}

\begin{exam}\label{ejemplo1}
The element $u=\left(\begin{array}{ccc}
                         1  & -1 & 1 \\
                         1  & -1 & 1 \\
                         1  & -1 & 1  \\
                 \end{array}\right)$ satisfies: $u^{2}=u$ and $x\cdot u=x$ for all $x\in \A$. The algebra of matrices $$\A=\left\{ \left(
              \begin{array}{ccc}
                x  &  -x  & x \\
                w  &  -w  & w  \\
                p  &  -p  & p  \\
              \end{array}
            \right)  : x, w, p \in \mathbb{R}
     \right\},$$ is an associative algebra under the usual matrix multiplication. $\A$  is a commutative algebra (non-associative) under the product $x \ast y = x \cdot y + y \cdot x$ ,
where $\cdot$ is the usual matrix multiplication. Then the linear map $R:\A\longrightarrow \A$ defined by

\begin{align*}
R\left(\left(
              \begin{array}{ccc}
                x  &  -x  & x \\
                w  &  -w  & w  \\
                p  &  -p  & p  \\
              \end{array}
            \right)  \right)=&  \left(
              \begin{array}{ccc}
                1  & -1  & 1 \\
                1  & -1  & 1  \\
                1  & -1  & 1  \\
              \end{array}
            \right)\cdot\left(
              \begin{array}{ccc}
                x  &  -x  & x \\
                w  &  -w  & w  \\
                p  &  -p  & p  \\
              \end{array}
            \right)  \\
                          =& (x-w+p)\left(
              \begin{array}{ccc}
                1  & -1  & 1 \\
                1  & -1  & 1  \\
                1  & -1  & 1  \\
              \end{array}
            \right)\\
\end{align*}
satisfies $R^{2}(x)=R(x)$  and $R(x)\ast R(y) = R(x\ast y)$ for all $x, y \in \A$. Therefore, we can define a Jordan algebra structure on $\A$ given by  $x \circ y = R(x) \ast R(y)$, and from it we can define a new Jordan algebra structure on $\A$ given by  $x \circ_{2} y = R(x) \circ y$.
\end{exam}

\section{Construction of  (left) Leibniz algebras from Associative Algebras with an Endomorphism Operator.}
Leibniz algebras were first introduced by J.-L. Loday in ~\cite{Loday1993} as a non-antisymmetric version of Lie algebras, and many results of Lie algebras have been extended to Leibniz algebras. Leibniz algebras play a significant role in different areas of mathematics and physics.

\begin{defn}
A (left) Leibniz algebra $\Lom$ is a vector space equipped with a bilinear map $$[\, , ]:\Lom\times \Lom\rightarrow \Lom$$ satisfying the (left) Leibniz identity
\begin{equation}\label{Ec:Leibniz}
[x,[y,z]]=[[x,y],z] + [y,[x,z]] \text{for all }  x,y,z \in \Lom.
\end{equation}
\end{defn}

\smallskip

\begin{rem}
We may pass from the right to the left Leibniz algebra by considering a new multiplication $x\circ y = [y, x]$. For a Leibniz algebra $\Lom$, we define left multiplication $l_{a} : \Lom \rightarrow \Lom$ by an element $a$ on an element $b$ by $l_{a}(b) = [a,b]$. Similarly, right multiplication by an element $a$ on an element $b$ is defined by $r_{a}(b) = [b,a]$. An algebra $\Lom$ over $\Field$ is a left Leibniz algebra if for every $x\in \Lom$ the corresponding operator $l_{x}$  of left multiplication is a derivation of $\Lom$, i.e., the mapping $l_{x}$ satisfies $l_{x}(a\cdot b) = l_{x}(a)\cdot b + a\cdot l_{x}(b)$,  $l_{x}\in Der(\Lom)$. Thus, left multiplication is a derivation in a left Leibniz algebra while right multiplication is not necessarily a derivation.
\end{rem}

\smallskip

\begin{prop}{~\cite{Vladimir2017} }
Let $\A$ be an associative algebra over a field $\Field$ and let $$R : \A \rightarrow \A$$ be an endomorphism of $\A$ such that $R^{2} = R$. Define the binary operation $[\, ·, ·]$ on $A$ by the following rule: $[a, b] = R(a)\cdot b - b\cdot R(a)$ for all elements $a,b \in \A$, Then, with respect to the operations $+$ and $[\, ·,·]$,  $\A$ becomes a Leibniz algebra.
\end{prop}

\begin{lem} \label{l:APPROX}
Let $\A$ be a associative algebra and let $R: \A\rightarrow \A$ be a linear map. Suppose that $R^{2}(x)=R(x) \text{ y }  R(x)\cdot R(y)=R(x\cdot y) \quad \text{for all } x,y\in\A.$
Then there exists a Leibniz structure on $\A$ given by
\begin{equation} \label{e:lei}
 [x,y]=R(x)\cdot y-R(y)\cdot R(x) \quad \text{for all } x,y\in\A.
 \end{equation}
\end{lem}

\begin{proof}
Let $x, y, z \in \A$, then we have\,
\begin{align*}
[[x,y],z]& = R([x,y])\cdot z-R(z)\cdot R([x,y]) \\
         & = R(R(x)\cdot y-R(y)\cdot R(x))\cdot z-R(z)\cdot R(R(x)\cdot y-R(y)\cdot R(x)) \\
         & = (R(x)\cdot R(y)-R(y)\cdot R(x))\cdot z-R(z)\cdot (R(x)\cdot R(y)-R(y)\cdot R(x)) \\
[y,[x,z]]& = R(y)\cdot [x,z]-R([x,z])\cdot R(y) \\
         & = R(y)\cdot ( R(x)\cdot z-R(z)\cdot R(x) )-R( R(x)\cdot z-R(z)\cdot R(x))\cdot R(y)  \\
         & = R(y)\cdot ( R(x)\cdot z-R(z)\cdot R(x) )-( R(x)\cdot R(z)-R(z)\cdot R(x))\cdot R(y)
\end{align*}
Then
\begin{align*}
[[x,y],z]+[y,[x,z]] & = (R(x)\cdot R(y)-R(y)\cdot R(x))\cdot z-R(z)\cdot (R(x)\cdot R(y)-R(y)\cdot R(x)) \\
                    &  \hspace{0.2cm} + R(y)\cdot ( R(x)\cdot z-R(z)\cdot R(x) )-( R(x)\cdot R(z)-R(z)\cdot R(x))\cdot R(y)
\end{align*}
so
\begin{align*}
[[x,y],z] + [y,[x,z]] & = (R(x)\cdot R(y))\cdot z -(R(x)\cdot R(z))\cdot R(y) - R(y)\cdot (R(z)\cdot R(x))  \\
                      & \hspace{0.5cm}+ R(z)\cdot (R(y)\cdot R(x))
\end{align*}
On the other hand, we have
\begin{align*}
[x,[y,z]]& = R(x)\cdot [y,z]-R([y,z])\cdot R(x) \\
         & = R(x)\cdot ( R(y)\cdot z-R(z)\cdot R(y) ) - R( R(y)\cdot z-R(z)\cdot R(y))\cdot R(x)  \\
         & = R(x)\cdot ( R(y)\cdot z-R(z)\cdot R(y) ) - (R(y)\cdot R(z)-R(z)\cdot R(y))\cdot R(x)
\end{align*}
Note that $R^{2}(x)=R(x) \text{ y }  R(x)\cdot R(y)=R(x\cdot y)$ for all $x, y \in \A$ implies $$[x,[y,z]]=[[x,y],z]+ [y,[x,z]]  \text{ for all }  x, y, z \in \A . $$
\end{proof}

\begin{exam}
The element $u=\left(\begin{array}{ccc}
                         -1  &1 & 1 \\
                         -1  &1 & 1 \\
                         -1  &1 & 1  \\
                 \end{array}\right)$ satisfies $u^{2}=u$ and $x\cdot u=x$ for all $x\in \A$. The algebra of matrices
$$\A=\left\{\left(
              \begin{array}{ccc}
                -y  & y  & y \\
                -m  & m  & m  \\
                -t  & t  & t  \\
              \end{array}
            \right):     y,m,t\in \mathbb{R}
     \right\} $$ is an associative algebra under the usual matrix multiplication.  Then the linear map $R:\A \longrightarrow \A$ defined by

\begin{align*}
R\left(\left(
              \begin{array}{ccc}
                -y & y & y \\
                -m & m & m \\
                -t & t & t \\
              \end{array}
            \right) \right)
         &=\left(\begin{array}{ccc}
                               -1&1 &1 \\
                               -1&1 &1 \\
                               -1&1 &1 \\
                 \end{array}\right)\cdot \left(
              \begin{array}{ccc}
                -y & y & y \\
                -m & m & m \\
                -t & t & t \\
              \end{array}
            \right)\\
        &=(-y+m+t)\left(\begin{array}{ccc}
                               -1&1 &1 \\
                               -1&1 &1 \\
                               -1&1 &1 \\
                        \end{array}\right)
\end{align*}
satisfies $R^{2}(x)=R(x)$ and $R(x)\cdot R(y) = R(x\cdot  y)$ for all $x, y \in \A.$  Therefore, we can define a Leibniz structure on $\A$ given by  $$[x, y] = R(x)\cdot y - R(y)\cdot R(x).$$
\end{exam}

\section{Construction of Pre-Lie algebras from Commutative Associative Algebras with a Endomorphism Operator}
In this section we present, in Proposition \ref{PreLie2} a construction of a Pre-Lie algebra structure given by $x\circ y=R(x)\cdot R(y)-y\cdot R(x)$ where $R$ is an endomorphism operator.

\begin{defn}
An algebra $\A$ over $\Field$ with a bilinear product $\circ$ which satisfies the following identity:
\begin{equation}\label{e:WeakBase}
(x\circ y)\circ z-x\circ (y\circ z)=  (y\circ x )\circ z-y\circ (x\circ z) \text{ for all }  x,y,z\in A
\end{equation} is called a Left Pre-Lie algebra.
\end{defn}

\begin{rem}
There is a construction of pre-Lie algebras using a commutative associative algebra $(\A,\cdot)$ with a derivation $D$ on $\A$,  the new product $a\ast b =a\cdot D(b)$, $\forall a,b\in \A$ makes $(\A, \ast)$ become a Novikov algebra, Novikov algebra is a pre-Lie algebra satisfying an additional identity: $(xy)z = (xz)y, \forall x,y,z \in \A$. There are some generalizations of the previous result for a commutative associative algebra $(\A,\cdot)$. If $D$  is a derivation on $\A$ , then the new product $x\ast_{a} y =x\cdot D(y)+a\cdot x \cdot y$, $\forall x,y \in \A$  makes $(\A, \ast_{a})$ become a Novikov algebra for a fixed element $a \in  \Field$ or $a \in \A$  (~\cite{Filipov1989, Xu1996}). In the case of an associative algebra $(\A,\cdot)$ with a Rota-Baxter relation of weight 1, it is known that if $x\ast y =R(x)\cdot y - y\cdot R(x) - x\cdot y$, $\forall x,y \in \A$, then the product $\ast$ defines a pre-Lie algebra on $\A$. (~\cite{Ebrahimi2002, Golubschik2000} )
\end{rem}

\begin{prop}\label{PreLie2}
Let $\A$ be a commutative associative algebra and let $R : \A \rightarrow  \A$  be a linear map such that $R^{2}(x)=R(x)$  and  $R(x)\cdot R(y) = R(x\cdot y)$ for all $x, y \in \A.$ Then we can define a Pre-Lie algebra structure on $\A$ given by
\begin{equation} \label{e:PreLi}
x\circ y=R(x)\cdot R(y)-y\cdot R(x) \hspace{0.2cm} (\text{respectively } x\circ y =R(x)\cdot y-R(y)\cdot R(x)).
\end{equation}
\end{prop}

\begin{proof}
Let $x,y,z\in \A$, then we have
\begin{align*}
(x\circ y)\circ z  & = (R(x)\cdot R(y)-y\cdot R(x))\circ z  \\
                   & = R(R(x)\cdot R(y)-y\cdot R(x))\cdot R(z)-z\cdot R(R(x)\cdot R(y)-y\cdot R(x)) \\
                   & = (R(x)\cdot R(y)-R(y)\cdot R(x))\cdot R(z)-z\cdot (R(x)\cdot R(y)-R(y)\cdot R(x)) \\
                   & = 0 \\
x\circ (y\circ z)  & = x\circ (R(y)\cdot R(z)-z\cdot R(y))\\
                   & = R(x)\cdot R(R(y)\cdot R(z)-z\cdot R(y))- (R(y)\cdot R(z)-z\cdot R(y))\cdot R(x) \\
                   & = R(x)\cdot(R(y)\cdot R(z)-R(z)\cdot R(y)) - (R(y)\cdot R(z)-z\cdot R(y))\cdot R(x)\\
                   & = -(R(y)\cdot R(z)-z\cdot R(y))\cdot R(x)
\end{align*}
On the other hand, we have
\begin{align*}
(y\circ x)\circ z  & = (R(y)\cdot R(x)-x\cdot R(y))\circ z  \\
                   & = R(R(y)\cdot R(x)-x\cdot R(y))\cdot R(z)-z\cdot R(R(y)\cdot R(x)-x\cdot R(y)) \\
                   & = (R(y)\cdot R(x)-R(x)\cdot R(y))\cdot R(z)-z\cdot (R(y)\cdot R(x)-R(x)\cdot R(y)) \\
                   & = 0 \\
y\circ (x\circ z)  & = y\circ (R(x)\cdot R(z)-z\cdot R(x))\\
                   & = R(y)\cdot R(R(x)\cdot R(z)-z\cdot R(x))- (R(x)\cdot R(z)-z\cdot R(x))\cdot R(y) \\
                   & = R(y)\cdot (R(x)\cdot R(z)-R(z)\cdot R(x)) - (R(x)\cdot R(z)-z\cdot R(x))\cdot R(y)\\
                   & = -(R(x)\cdot R(z)-z\cdot R(x))\cdot R(y)
\end{align*}

Therefore, $(x\circ y)\circ z-x\circ (y\circ z) =  (y\circ x )\circ z-y\circ (x\circ z).$
\end{proof}

\subsection{Examples of Pre-Lie algebras from commutative associative algebras with the endomorphism operator.}

In this subsection we present in the example \ref{PreLieC} a commutative associative algebra $\A$ with the Hadamard or entrywise matrix multiplication, that allow us to give an example of Pre-lie algebras from a endomorphism operator with certain properties on the algebra  $\A$.

\begin{exam}\label{PreLieC}
Let $\A$ be the vector space of all $2\times 2$ matrices over $\mathbb{R},$

$$\A=\left\{\left(\begin{array}{ccc}
                m  & x  \\
                n  & y  \\
              \end{array}\right): m,n,x, y \in \mathbb{R}  \right\}.$$  $\A$  is a commutative associative algebra with the unusual matrix multiplication defined by:
\begin{equation}\label{Mult}
      \left(\begin{array}{ccc}
                m  & x  \\
                n  & y  \\
              \end{array}\right)\ast \left(
              \begin{array}{ccc}
                p  & z  \\
                r  & w  \\
              \end{array}
      \right)= \left(
              \begin{array}{ccc}
                mp     & xz  \\
                nr     & yw  \\
              \end{array}
      \right).
\end{equation}

The element $I=\left(\begin{array}{ccc}
                1  & 1  \\
                1  & 1  \\
              \end{array}\right)$  is the multiplicative identity.
\end{exam}

\begin{exam}
The element $u=\left(\begin{array}{ccc}
                1  & 0  \\
                0  & 0  \\

              \end{array}\right)$  with the usual matrix multiplication satisfies: $u^{2}=u$,  $u\cdot x \in \A$ and $x\cdot u=x$ for all $x\in \A$. The algebra of matrices
$$\A=\left\{\left(
              \begin{array}{ccc}
                m  & 0  \\
                n  & 0  \\
              \end{array}
            \right):  m,n\in\mathbb{R}
      \right\} $$ is a commutative associative algebra under the unusual matrix multiplication defined in \ref{Mult}. Then the linear map $R:\A \longrightarrow \A$ defined by

\begin{align*}
R\left( \left(
              \begin{array}{ccc}
                m  & 0  \\
                n  & 0  \\
              \end{array}
            \right) \right)=  & \left(\begin{array}{ccc}
                1  & 0  \\
                0  & 0  \\
              \end{array}\right)\cdot \left(
              \begin{array}{ccc}
                m  & 0  \\
                n  & 0  \\
              \end{array}
            \right) \\
       =  & \left(\begin{array}{ccc}
                m  & 0  \\
                0  & 0  \\
              \end{array}\right)
\end{align*}
satisfies $R^{2}(x)=R(x)$ and $R(x)\ast  R(y) = R(x\ast   y)$ for all $x, y \in \A.$  Therefore, we can define a Pre-Lie algebra structures on $\A$ given by  $$x\circ y = R(x)\ast  R(y) - y\ast  R(x).$$
\end{exam}

\section{Construction of Pre-Lie algebras from Commutative Associative Algebras with a Differential Operator }

We now present in the Propositions \ref{PreLie1}\, a construction of Pre-Lie algebra structure given by $x\circ y=R(x)\cdot y$ where $R$ is a differential operator, and we also introduce the notion of {\it left zero divisor on $\A$}: Let $( \, \h, \, \cdot \, )$ be a set $\h$ with a binary operation $\cdot$ on it, then an element $u$ of $\h$ is called a left zero divisor on $\A \subseteq \h$ if $x \cdot u = 0$ for all $x$ in $\A$.

\begin{prop}\label{PreLie1}
Let $\A$ be a commutative associative algebra and let $R : \A \rightarrow  \A$ be a linear
map such that $R^{2}(x)=\alpha.\,x$  and  $R(x)\cdot y+ x\cdot R(y) = R(x\cdot y)$ for all $x, y \in \A.$
Then we can define a Pre-Lie algebra structures on $\A$ given by $x\circ y=R(x)\cdot y.$
\end{prop}

\begin{proof}
Let $x,y,z\in \A$, then
\begin{align*}
  (x\circ y)\circ z-x\circ (y\circ z)  & =  (R(x)\cdot y)\circ z-x\circ (R(y)\cdot z)\\
                                       & =  R(R(x)\cdot y)\cdot z-R(x)\cdot(R(y)\cdot z)\\
                                       & =  (R^{2}(x)\cdot y + R(x)\cdot R(y))\cdot z-R(x)\cdot(R(y)\cdot z) \\
                                       & =  (R^{2}(x)\cdot y)\cdot z\\
                                       & = ((\alpha.\,x)\cdot y)\cdot z .
\end{align*}
On the other hand, we have
\begin{align*}
                            (y\circ x )\circ z-y\circ (x\circ z) & =  (R(y)\cdot x)\circ z-y\circ (R(x)\cdot z)\\
                                                                 & =  R(R(y)\cdot x)\cdot z-R(y)\cdot(R(x)\cdot z) \\
                                                                 & = (R^{2}(y)\cdot x + R(y)\cdot R(x))\cdot z-R(y)\cdot(R(x)\cdot z)\\
                                                                 & = (R^{2}(y)\cdot x)\cdot z \\
                                                                 & = ((\alpha.\,y)\cdot x)\cdot z .
\end{align*}
Therefore, $(x\circ y)\circ z-x\circ (y\circ z) =  (y\circ x )\circ z-y\circ (x\circ z).$
\end{proof}

\subsection{Examples of Pre-Lie algebras from commutative associative algebras with a Differential Operator.}

\begin{prop}
Let $( \, \A, \, \cdot \, )$ be an associative subalgebra of an algebra $\h$, and suppose that there exists $u\in \h$ such that $u\cdot x  \in \A$ and $x\cdot u=0$ for all $x\in \A$. Then the linear map $R:\A\longrightarrow \A$ defined by $R(x)=u\cdot x$ satisfies
\begin{equation} \label{e:opedif}
R(x)\cdot y+x\cdot R(y)=R(x\cdot y) \text{ for all } x,y \in \A.
\end{equation}
\end{prop}

\begin{proof}
Let $x,y \in \A$, then
\begin{align*}
R(x)\cdot y+x\cdot R(y) &=(u\cdot x)\cdot y + x\cdot(u\cdot y)\\
                        &= u\cdot(x\cdot y) + (x\cdot u)\cdot y \\
                        &= u\cdot (x\cdot y)+ (0\cdot y)= u\cdot(x\cdot y).
\end{align*}
On the other hand, $R(x\cdot y)=u\cdot (x\cdot y).$ Therefore $R(x)\cdot y+x\cdot R(y) = R(x\cdot y)$ for all $x, y \in \A.$
\end{proof}

\begin{exam}
We consider the subalgebra
$$\A=\left\{\left(\begin{array}{ccc}
                   y & y & 0 \\
                   n & n & 0 \\
                   r & r & 0 \\
                 \end{array}\right): r, n , y \in \mathbb{R}  \right\}$$  under the usual matrix multiplication.
The element $u=\left(\begin{array}{rrr}
                   a    &  b  & c \\
                   -a   &  -b & -c \\
                    e   &  f  &  g\\
                 \end{array}\right)$ satisfies $x\cdot u=0$ and $u\cdot x\in \A$ for all $x\in \A$.

Then the linear map $R:\A\longrightarrow \A$ defined by

\begin{align*}
 R\left( \left(\begin{array}{rrr}
                   y    &  y  & 0 \\
                   n   &   n  & 0 \\
                   r   &   r  & 0\\
 \end{array}\right) \right)           & = \left(\begin{array}{rrr}
                                              a    &  b  & c \\
                                             -a   &  -b & -c \\
                                              e   &  f  &  g\\
                                          \end{array}\right)\cdot \left(\begin{array}{rrr}
                                                                     y    &  y  & 0 \\
                                                                     n   &   n & 0 \\
                                                                     r   &   r  & 0\\
                                                                   \end{array}\right)\\
                                      & = \left(\begin{array}{rrr}
                                              ay+bn+cr    &  ay+bn+cr  & 0 \\
                                              -ay-bn-cr   &  -ay-bn-cr & 0 \\
                                              ey+fn+gr    &   ey+fn+gr  & 0\\
                                           \end{array}\right)  \\
\end{align*}

satisfies $R(x)\cdot y+x\cdot R(y) = R(x\cdot y)$ for all $x, y \in \A.$
\end{exam}

\begin{exam}
The algebra of matrices
$$\A=\left\{\alpha\left(\begin{array}{ccc}
                   an & ap & aq \\
                   bn & bp & bq \\
                    n &  p &  q \\
                 \end{array}\right): \alpha \in \mathbb{R}\right\}$$ where $ a=-\beta b \text{ and }  b,n,p,q \in \mathbb{R}$ is a commutative associative algebra.

The element $u=\left(\begin{array}{cll}
                   a    &  \beta a   & \lambda\beta a \\
                   b    &  \beta b   & \lambda\beta b \\
                   1    &  \beta     &  \lambda\beta  \\
                 \end{array}\right)$ where $\lambda,\beta \in \mathbb{R}$,  satisfies \newline $u^{2}=(\lambda\beta)u$, \, $x\cdot u=0 \, ( \Leftrightarrow an+bp+q=0 )  $ and $u\cdot x\in \A$ for all $x\in \A$. \\
Then the linear map $R:\A\longrightarrow \A$ defined by
\begin{align*}
R\left(\left(\begin{array}{rrr}
                  an    &  ap  & aq \\
                   bn   &   bp  & bq \\
                    n   &    p  &  q\\
                 \end{array}\right) \right)
                 & =\left(\begin{array}{ccr}
                   a    &  \beta a   & \lambda\beta a \\
                   b    &  \beta b   & \lambda\beta b \\
                   1    &  \beta    & \lambda\beta \\
                 \end{array}\right)\cdot \left(\begin{array}{rrr}
                   an    &  ap  & aq \\
                   bn    &  bp  & bq \\
                   n     &  p    & q\\
                 \end{array}\right)\\
                 &=(\lambda\beta )\left(\begin{array}{ccc}
                    an           &  ap   & aq \\
                    bn           &  bp   & bq \\
                     n           &   p   &  q\\
                 \end{array}\right)\\
\end{align*}
satisfies $R^{2}(x)=(\lambda\beta)R(x)$ and $R(x)\cdot y + x\cdot R(y)=R(x\cdot y)$ for all $x, y \in \A.$ Therefore, we can define a Pre-Lie algebra structures on $\A$ given by $x\circ y=R(x)\cdot y.$
\end{exam}

\section{Construction of flexible algebras from Associative Algebras with a left averaging operator.}

\begin{defn}
A flexible algebra is a vector space $\J$ over a field $\Field$ of characteristic $\neq 2$  with a binary operation $\circ$ satisfying for $x,y\in \J$ the following identity:
\begin{equation}\label{Ec:flexible}
(x\circ y)\circ  x = x\circ (y \circ x).
\end{equation}
\end{defn}

\begin{rem}
The study flexible algebras was initiated by Albert (\cite{ALBERT1948}) and investigated by the authors Myung, Okubo, Laufer, Tomber and Santilli, see for example (\cite{MYUNG1982}).
\end{rem}

\begin{prop}\label{p:ESSPECT}
Suppose $(\A, \cdot )$ is a flexible algebra, and $R : \A \rightarrow  \A$ is a linear
map, such that $R^{2}=R$ and
\begin{equation} \label{e:ESSINEQ}
   R(x)\cdot R(y) = R(R(x)\cdot y)= R(x \cdot y) \text{ for all x, y } \in \A .
\end{equation}
Then we can define a new flexible algebra structures on $\A$ given by $$x\circ y=R(x\cdot y).$$

\end{prop}

\begin{proof}
Let $\A$  be a flexible algebra and  $R : \A \rightarrow  \A$ is a linear map such that  $R(x)\cdot R(y)=R(R(x)\cdot y)=R(x\cdot y)$ for all $x, y \in \A.$ So $(x\circ y) \circ x = (R(x\cdot y))\circ x =R(R(x\cdot y)\cdot x )= (R(x)\cdot R(y))\cdot R(x).$  Since $ x\circ (y \circ x)  = x\circ ( R(y\cdot x)) = R(x\cdot R(y\cdot x))= R(x)\cdot R( R(y\cdot x))= R(x)\cdot  R(y\cdot x)= R(x)\cdot  (R(y)\cdot R(x))$ and $(\A, \cdot )$ is a flexible algebra, then we have $(x\circ y)\circ  x = x\circ (y \circ x)$  for all $x, y \in \A.$ Therefore $(\A, \circ )$ is a flexible algebra.
\end{proof}

\subsection{Examples of Flexible algebras from associative algebras with a left averaging operator.}

\begin{prop} \label{p:wAPPROX}
Let $(\, \A, \, \cdot \,)$ be an associative subalgebra of an algebra $\h$. Suppose that $\h$ has the following property:
\begin{quote}
  {There exists $u\in \h,$ such that $u\cdot x \in \A$ and $x\cdot u=u\cdot x$ for all $x\in \A$.}
\end{quote}
Then the linear map $R:\A\longrightarrow \A$ defined by $R(a)= u\cdot x $ satisfies the left averaging identity
\begin{equation} \label{e:averaging}
R(a)\cdot R(b)=R(R(a)\cdot b) \text{ for all } a,b \in \A.
\end{equation}
Furthemore, if $u^{2}=u.$ Then
\begin{equation} \label{e:averagingEnd}
R(a)\cdot R(b)=R(R(a)\cdot b)=R(a\cdot b) \text{ for all } a,b \in \A.
\end{equation}
\end{prop}

\begin{proof}
Let $x,y \in \A$, by the hypothesis there exists $u\in \h$ such that $u\cdot x  \in \A$  and $x\cdot u=u\cdot x$ for all $x\in \A$, then
\begin{align*}
R(x)\cdot R(y)  &= (u\cdot x)\cdot(u\cdot y) \\
                &= u\cdot((x\cdot u)\cdot y) \\
                &= u\cdot((u\cdot x)\cdot y) \\
                &= R(R(x)\cdot y).
\end{align*}
If $u^{2}=u,$ then $R(x)\cdot R(y)=u\cdot((u\cdot x)\cdot y)=u^{2}\cdot( x\cdot y)=u\cdot( x\cdot y)=R(x\cdot y).$ Therefore $R(x)\cdot R(y) = R(R(x)\cdot y) = R(x\cdot y)$ for all $x, y \in \A.$
\end{proof}

\begin{exam}

We consider the subalgebra
$$ \A=\left\{\left(\begin{array}{ccc}
                   y & n & 0 \\
                   0 & y & 0 \\
                   0 & 0 & r \\
                 \end{array}\right): r, n , y \in \mathbb{R}  \right\}$$  under the usual matrix multiplication.

The element $u=\left(\begin{array}{ccr}
                   1    &  0   &  0 \\
                   0    &  1   &  0 \\
                   0    &  0   &  0\\
                 \end{array}\right)$ satisfies $u^{2}=u$, $x\cdot u=u\cdot x$ and $u\cdot x\in \A$ for all $x\in \A$. Then the linear map $R:\A\longrightarrow \A$ defined by
$$R\left( \left(\begin{array}{rrr}
                   y    &  n  & 0 \\
                   0   &   y  & 0 \\
                   0   &   0  & r\\
                 \end{array}\right) \right)=\left(\begin{array}{ccr}
                   1    &  0    & 0 \\
                   0    &  1    & 0 \\
                   0   &   0    & 0\\
                 \end{array}\right)\cdot \left(\begin{array}{rrr}
                   y    &   n  & 0 \\
                   0    &   y  & 0 \\
                   0    &   0  & r \\
                 \end{array}\right)=\left(\begin{array}{rrr}
                   y    &   n  & 0 \\
                   0    &   y  & 0 \\
                   0    &   0  & 0 \\
                 \end{array}\right)$$
satisfies $R^{2}=R$ and $R(x)\cdot R(y) = R(R(x)\cdot y)= R(x\cdot y)$ for all $x, y \in \A.$  Therefore, we can define a flexible algebra structures on $\A$ given by $x\circ y=R(x\cdot y).$ We also have that $\A$ is a Lie algebra with the product $[x,y]=x\cdot R(y)-y\cdot R(x)$ (see Proposition \ref{p:LieAlge}), Therefore $(\A, [ ,]) $ is a flexible algebra.
\end{exam}

\section{Construction of Rota-Baxter Operator}

In this section we present, in the Propositions \ref{PesoUno}, constructions of Rota-Baxter Operators of weight $\lambda=1$ and $\lambda=0$  from associative algebra with an element $u$ skew-idempotent or nilpotent of index 2 respectively, and we also introduce the notion of {\it Rota-Baxter Operator of weight $(\lambda , \beta)$}.\\

We recall from the Introduction:

\begin{defn}
Let $(\A,\cdot)$ be an associative algebra. A linear map $R:\A\rightarrow \A$ is called a Rota-Baxter operator of weight $\lambda$ on $\A$ if $R$
satisfies
\begin{equation}\label{Ec:weigh}
R(x)\cdot R(y)=R\left(R(x)\cdot y+x\cdot R(y)+\lambda\,  x\cdot y\right),
\end{equation}
for all $x,y\in \A$. A Rota-Baxter algebra (also known as a Baxter algebra) is an associative algebra $\A$ with a Rota-Baxter operator.
\end{defn}

\begin{rem}
One importance of the Rota-Baxter Algebra is its close relationship with other algebraic structures. For example pre-Lie algebras come naturally from a Rota Baxter-Operator on an Lie Algebras. ~\cite{Golubchik2000}, ~\cite{Medina1981} .
\end{rem}

\begin{defn}
An element $u\neq 0$  of an algebra $\A$ is called nilpotent if $u^{n}=0$  for some integer $n$ . The least such integer is called the index of $u$.
\end{defn}

\begin{defn}
An element $u\neq 0$  of an algebra $\A$  is said to be skew-idempotent with respect to a product $\cdot$ in the algebra $\A$  if $u\cdot u = -u$.
\end{defn}

\begin{prop}\label{PesoUno}
Let $(\, \A, \,  \cdot \,)$ be an associative algebra and suppose that there exists $u\in \A$ such that $u^{2}=-u$  and $u\cdot x\in \A$ for all $x\in \A$. Then the linear map $R:\A\longrightarrow \A$ defined by $R(x)=u\cdot x$ satisfies
\begin{equation} \label{e:RBO1 }
R\left(\, R(x)\cdot y+ x\cdot R(y) + x\cdot y \, \right)=R(x)\cdot R(y) \text{ for all } x,y \in \A.
\end{equation}
Furthemore, if $u^{2}= 0$, then $R$ is a Rota-Baxter operator of weight zero on $\A$.
\end{prop}

\begin{proof}
Let $x,y\in \A,$ then we have $R(x)\cdot R(y)= (u\cdot x)\cdot (u\cdot y).$ On the other hand,
\begin{align*}
R(R(x)\cdot y+x\cdot R(y) + x\cdot y)  & = R((u\cdot x)\cdot y+x\cdot(u\cdot y)+ x\cdot y ) \\
                                       & =u\cdot((u\cdot x)\cdot y+x\cdot(u\cdot y)+ x\cdot y )\\
                                       & = u^{2}\cdot ( x\cdot y)+(u\cdot x)\cdot(u\cdot y)+ u\cdot (x\cdot y)\\
                                       & =(u^{2}+u)\cdot ( x\cdot y) +(u\cdot x)\cdot(u\cdot y).
\end{align*}
Therefore $R(R(x)\cdot y+x\cdot R(y) + x\cdot y)=R(x)\cdot R(y) \text{ for all } x,y \in A.$  Now, if $u^{2}= 0$, then
\begin{align*}
R(R(x)\cdot y+x\cdot R(y)) &= (u\cdot x)\cdot(u\cdot y).
\end{align*}
Therefore $R(R(x)\cdot y+x\cdot R(y)) = R(x) \cdot R(y)$ for all $x, y \in \A.$
\end{proof}

\begin{exam}
The element $u=\left(\begin{array}{cc}
                   xy & -x^{2} \\
                   y^{2} & -xy \\
                 \end{array}\right)$ satisfies $u^{2}=0$ and $u\cdot x\in \A$ for all $x$ in the algebra of matrices $$\A=\left\{\left(\begin{array}{cc}
                   0 & a \\
                   0 & b \\
                 \end{array}\right): a,b \in \mathbb{R}  \right\}$$  considered as a subalgebra of $B=M_{2\times 2}$ under the usual matrix multiplication. If we define the map $R:\A\longrightarrow \A$ by
                $$R(\left(\begin{array}{cc}
                   0 & a \\
                   0 & b \\
                 \end{array}\right))=\left(\begin{array}{cc}
                   xy & -x^{2} \\
                   y^{2} & -xy \\
                 \end{array}\right)\left(\begin{array}{cc}
                   0 & a \\
                   0 & b \\
                 \end{array}\right)=\left(\begin{array}{cc}
                   0 & xya-x^{2}b \\
                   0 & y^{2}a-xyb \\
                 \end{array}\right),$$ then $R$ satisfies $ R(R(x)\cdot y+y\cdot R(x)) = R(x)\cdot R(y);$  for all $x, y \in \A.$
\end{exam}

\begin{exam}\label{ejemplo1}
The element $u=\left(\begin{array}{cc}
                   x & y \\
                   \frac{-x^{2}-x}{y} & -x-1 \\
                 \end{array}\right)$  , $y\neq 0$, is a skew-idempotent in the algebra of matrices under the usual matrix multiplication, that is, $u^{2}=-u$. We observe that $u\cdot x\in \A$ for all $x\in \A$.
If we define the map $R:\A\longrightarrow \A$ by $$R(\left(\begin{array}{cc}
                   0 & a \\
                   0 & b \\
                 \end{array}\right))=\left(\begin{array}{cc}
                   x & y \\
                   \frac{-x^{2}-x}{y} & -x-1 \\
                 \end{array}\right)\left(\begin{array}{cc}
                   0 & a \\
                   0 & b \\
                 \end{array}\right)=\left(\begin{array}{cc}
                   0 & xa+yb \\
                   0 & (\frac{-x^{2}-x}{y})a -(x+1)b \\
                 \end{array}\right)$$
then $R$ satisfies $R(x)\cdot R(y) = R(R(x)\cdot y + x \cdot R(y)+x\cdot y);$  for all $x, y \in \A.$
\end{exam}

\begin{defn}\label{Def RBO: CAP4}
Let $(\A,\cdot)$ be an associative algebra. A linear map $R:\A\rightarrow \A$ is called a Rota-Baxter operador of weight $(\lambda, \beta) $ on $\A$ if $R$ satisfies
\begin{equation}\label{Ec:RBO}
R(x)\cdot R(y)=R\bigl(R(x)\cdot y+x\cdot R(y)+\lambda\, x\cdot y\bigr)+\beta\, x\cdot y,  \text{ for all } x,y\in\A.
\end{equation}
\end{defn}

\begin{rem}
A Rota-Baxter operador of weight $(\lambda, \beta) $  for associative algebras allows to build examples of Dyck$^{m}$-algebras  ~\cite{Martinez2021}. The main result of this section is the construction of Rota-Baxter operador of weight $(\lambda,\beta) $ on associative algebras.
\end{rem}

\begin{prop}\label{ORB2}
Let $(\, \A, \,  \cdot \,)$ be an associative algebra and suppose that there exists $u\in \A$ such that $u^{2}=-\lambda u - \beta 1_{A}$  and $u\cdot x\in \A$ for all $x\in \A$. Then the linear map $R:\A\longrightarrow \A$ defined by $R(x)=u\cdot x$ satisfies
\begin{equation} \label{e:RBOG}
R\left(\, R(x)\cdot y+ x\cdot R(y) +\lambda x\cdot y \, \right)+\beta x \cdot y =R(x)\cdot R(y) \text{ for all } x,y \in \A.
\end{equation}
\end{prop}

\begin{proof}
Let $x,y\in \A,$ then we have
\begin{align*}
R(R(x)\cdot y+x\cdot R(y) + \lambda x\cdot y)+\beta x\cdot y  & = R((u\cdot x)\cdot y+x\cdot(u\cdot y)+ \lambda x\cdot y )+\beta x\cdot y \\
                                                              & = u\cdot((u\cdot x)\cdot y+x\cdot(u\cdot y)+\lambda x\cdot y )+\beta x\cdot y\\
                                                              & = u^{2}\cdot ( x\cdot y)+(u\cdot x)\cdot(u\cdot y)+ \lambda u\cdot (x\cdot y)+\beta x\cdot y \\
                                                              & = (u^{2}+\lambda u+\beta 1_{A} )\cdot ( x\cdot y) +(u\cdot x)\cdot(u\cdot y)\\
                                                              & = (u\cdot x)\cdot(u\cdot y)
\end{align*}
On the other hand,  $R(x)\cdot R(y)=(u\cdot x)\cdot (u\cdot y).$ Therefore, $$R(R(x)\cdot y+x\cdot R(y) + \lambda x\cdot y)+\beta x\cdot y=R(x)\cdot R(y) \text{ for all } x,y \in \A $$
\end{proof}

\begin{exam}\label{ejemplo1}
The element $u=\left(\begin{array}{cc}
                   x & y \\
                   \frac{-x^{2}-\lambda x-\beta }{y} & -x-\lambda \\
                 \end{array}\right)$, where $y\neq 0$, satisfies $u^{2}=-\lambda u- \beta 1_{A}$ and $u\cdot x\in \A$ for all $x\in \A$.
If we define the map $R:\A\longrightarrow \A$ by
\begin{align*}
R(\left(\begin{array}{cc}
                   0 & a \\
                   0 & b \\
         \end{array}\right)) & =\left(\begin{array}{cc}
                                        x & y \\
         \frac{-x^{2}-\lambda x-\beta}{y} & -x-\lambda \\
                                     \end{array}\right)\left(\begin{array}{cc}
                                                                  0 & a \\
                                                                  0 & b \\
                                                                 \end{array}\right)\\
                            &  =\left(\begin{array}{cc}
                                       0 & xa+yb \\
                                       0 & (\frac{-x^{2}-\lambda x-\beta}{y})a -(x+\lambda)b \\
                                      \end{array}\right)\\
\end{align*}
then $R$ satisfies $R(x)\cdot R(y) = R(R(x)\cdot y+x\cdot R(y) + \lambda x\cdot y)+\beta x\cdot y,$  for all $x, y \in \A.$
\end{exam}

\subsection*{Acknowledgment}
We thank to Universidad del Cauca for the support to our research group “Estructuras Algebraicas, Divulgación Matemática y Teorías Asociadas. @DiTa” under the research project with ID 5773, entitled “Aplicaciones de Estructuras Algebraicas". We also thank to the anonymous referees for their helpful comments. This work is dedicated to my daughters, especially to Clara Isabel Martinez Ceron (January 17, 2017).

\bibliographystyle{amsplain}
\bibliography{xbib}

\end{document}